\documentclass{article}
\usepackage{amsmath,amssymb,amsthm}
\newtheorem{thm}{Theorem}[section]
\newtheorem{pro}[thm]{Proposition}
\newtheorem{lem}[thm]{Lemma}

\newtheorem{cor}[thm]{Corollary}

\theoremstyle{definition}

\newtheorem{defi}[thm]{Definition}

\begin{document}
\date{}
\title{\bf  The Terwilliger algebra of the Odd graph revisited from the  viewpoint of group action}

\author{Lihang Hou$^{\rm a}$\, \ Suogang Gao$^{\rm b}$\, Na Kang$^{\rm a}$\,\ Bo Hou$^{\rm b,}$\thanks{Corresponding author. E-mail address: houbo1969@163.com.}\\
{\footnotesize  $^{\rm a}$ School of Mathematics and Science, Hebei GEO University, Shijiazhuang, 050031, P. R. China}\\
\footnotesize $^{\rm b}$ School of Mathematical Sciences, Hebei Normal University, Shijiazhuang, 050024, P. R. China}
\maketitle
\begin{abstract}
Let $O_{m+1}$ denote the Odd graph  on a set of cardinality $2m+1$, where $m$ is a positive integer. Denote by   $X$ its vertex set and
by $T:=T(x_0)$ its Terwilliger algebra  with respect to any fixed vertex $x_0\in X$.
In this paper, we first prove that $T$ coincides with the centralizer algebra
of the stabilizer of $x_0$ in the automorphism group of $O_{m+1}$ by considering
the action of this automorphism group on $X\times X\times X$.
Then we give the decomposition of $T$ for $m\geq 3$ by using all the homogeneous components of $V:=\mathbb{C}^X$, each of which is a nonzero subspace of $V$ spanned by the  irreducible $T$-modules that are isomorphic.
Finally, we display an orthogonal  basis  for every  homogeneous component of $V$.
\end{abstract}



\section{Introduction}

The Terwilliger (or subconstituent) algebra of a commutative association scheme was first introduced in \cite{p2}. This algebra is a finite-dimensional, semisimple $\mathbb{C}$-algebra and is non-commutative in general. The Terwilliger algebra has been successfully used for  studying the commutative association schemes and (in particular, $Q$-polynomial)  distance-regular graphs.

Let $\Gamma=(Y, \mathcal{R})$ denote a distance-regular graph with vertex set $Y$ and
edge set $\mathcal{R}$. The Terwilliger algebra of $\Gamma$ with respect to a fixed vertex $x\in Y$
is in fact a subalgebra of the centralizer algebra  of the stabilizer of $x$ in  automorphism group of $\Gamma$ (cf. \cite{tyy}). The latter algebra is defined to be the set of matrices that are
invariant under permuting the rows and columns by elements of this stabilizer.
 In general, the above two algebras  do not coincide.  However,
it was proved that the two algebras are equal for some distance-regular graphs:
the Hamming graphs (\cite{g1,s}), the Johnson graphs (\cite{tyy}), the folded $n$-cubes (\cite{hhgy}), the halved $n$-cubes (\cite{hhkg}) and the halved folded $2n$-cubes (\cite{cckh}). An interesting and important problem is to find when the two algebras coincide.

The present paper is about the Terwilliger algebra  of the Odd graphs known as a class of  distance-transitive graphs.
Let $O_{m+1}$ denote the Odd graph with vertex set $X$ and let $V:=\mathbb{C}^{X}$ denote the column vectors space with coordinates indexed by $X$.
Let $T:=T(x_0)$ denote the Terwilliger algebra of $O_{m+1}$ with respect to  fixed vertex $x_0\in X$; note that the distance-transitivity  of $O_{m+1}$ makes the algebra $T(x_0)$ independent of the choice of $x_0$.

We first prove that $T$ and the corresponding centralizer algebra are also equal for $O_{m+1}$
by considering
the action of the automorphism group of $O_{m+1}$ on $X\times X\times X$, and consequently obtain a basis of $T$ (see Theorem \ref{thm1}, Corollary \ref{cor2}).

To describe  the structure of irreducible modules is also an important problem in  study of the Terwilliger algebra.  We  will therefore focus on the irreducible $T$-modules for $m\geq 3$ after giving the first main result. J.S. Caughman $et\ al$. \cite{cau} and P. Terwilliger \cite[Example 6.1]{ter3}  characterized some properties of irreducible $T$-modules.
These properties can tell us that the isomorphism class
of an irreducible $T$-module  depends only on two parameters, called dual endpoint and diameter. Based on these properties, we further give the feasible dual endpoints and diameters for all irreducible $T$-modules so that we obtain all isomorphism classes
of irreducible $T$-modules (see Theorem \ref{lem5}). Moreover, a decomposition of $T$ (see Theorem \ref{pro5}) is also given  by using all the homogeneous components of $V$, each of which is a nonzero subspace of $V$ spanned by the  irreducible $T$-modules that are $T$-isomorphic; this work is originally motivated by the fact that $T$ is isomorphic to a  direct sum of full matrix algebras.

Finally,
we display an orthogonal  basis for every  homogeneous component of $V$  by using the basis of $T$ from Corollary \ref{cor2} and some properties of irreducible $T$-modules (see Theorem \ref{pro2}).

We remark that (i) for the Johnson graph $J(N;m)$ with $2m\leq N$, Y. Tan $et\ al$. \cite{tyy} showed that the Terwilliger algebra is isomorphic to the centralizer algebra  from the viewpoint of group representations. Whereas the $O_{m+1}$ is the distance-$m$ graph of the
$J(2m+1;m)$ and they have the same Terwilliger algebra. (ii) Q. Kong $et\ al$. \cite{klw} determined a basis of $T$  by using the Kronecker product
of intersection matrices and obtained the  dimension of $T$. However, the basis of $T$ given in this paper is different from that determined in \cite{klw}; (iii) the decomposition of $T$ in Theorem \ref{pro5}  of this paper might be
useful in coding theory to derive code upper bounds via semidefinite programming.

We end the introduction by recalling the definition of Odd graphs. Let $S=\{1,2,\ldots, 2m+1\}$ for a positive integer $m$.  Denote by $X$  the  collection of all $m$-subsets of $S$.
The Odd graph on $S$ is described as the graph
whose vertex set is $X$, where two $m$-subsets are adjacent whenever they are disjoint.
It is not difficult to verify that the path-length distance between $x$ and $y$ is given by
\begin{align}\label{eq1}
\partial(x,y)=\left\{\begin{array}{ll} 2|x\cap y|+1 &\text{if $|x\cap y|\leq \lfloor \frac{m-1}{2}\rfloor$},\\[0.2cm]
2m-2|x\cap y| &\text{if $|x\cap y|\geq \lfloor \frac{m-1}{2}\rfloor+1$}
\end{array}\right.\ \ \ \  (x,y\in X),
\end{align}
where $\lfloor a\rfloor$ denotes the maximal integer less than or equal to $a$. The above Odd graph, often denoted by $O_{m+1}$, is an almost-bipartite $Q$-polynomial distance-regular graph with diameter $m$. For more information on $O_{m+1}$, we refer to \cite{bcn}.

\section{Preliminaries}

In this section, we recall some  concepts and basic facts concerning
distance-regular graphs and Terwilliger algebras.

Let $\mathbb{C}$ denote the complex number field  and let $Y$ denote a nonempty finite set.
Let $V:=\mathbb{C}^{Y}$ denote the column vectors space with coordinates indexed by $Y$. We
endow $V$ with the Hermitian inner product $<, >$ that satisfies $<u,v>=u^{\rm t}\overline{v}$ for $u, v\in V$, where ${\rm t}$ denotes
transpose and ${}^-$ denotes complex conjugation. Let  ${\rm{Mat}}_{Y}(\mathbb{C})$ denote the $\mathbb{C}$-algebra of matrices with rows and columns indexed by $Y$. Observe that ${\rm{Mat}}_{Y}(\mathbb{C})$ acts on $V$ by left multiplication naturally;
we call $V$ the {\it standard module}.

Let $\Gamma=(Y, \mathcal{R})$ denote a finite, undirected, connected graph,
without loops or multiple edges, with vertex set $Y$,
edge set $\mathcal{R}$, path-length distance function $\partial$, and diameter $D:=\text{max}\{\partial(x,y)\mid x,y\in Y\}$. We say $\Gamma$ is {\it{distance-regular}} whenever for all integers $i, j,h\ (0\leq i, j,h\leq D)$ and for all vertices $x, y\in Y$
such that $\partial(x, y)=h$, the number $p^h_{ij}(x,y)=|\{z\in Y \mid \partial(x, z)=i, \partial(z, y)=j\}|$
is independent of $x$ and $y$. The constants
$p^h_{ij}:=p^h_{ij}(x,y)$ are called the {\it intersection numbers} of $\Gamma$.
 Next, we assume $\Gamma$ is distance-regular.

For $0\leq i\leq D$, let $A_i\in {\rm{Mat}}_{Y}(\mathbb{C})$ denote the $i$-th {\it{distance matrix}} of $\Gamma$: the $(x,y)$-entry of $A_i$ is $1$ if $\partial(x,y)=i$ and $0$ otherwise. Observe (\rm{i}) $A_0=I$; (\rm{ii}) $\overline{A}^{\rm t}_i=A_i\ (0\leq i \leq D)$; (\rm{iii}) $\sum_{i=0}^DA_i=J$;
(\rm{iv}) $A_iA_j=\sum^D_{h=0}p^h_{ij}A_h\ (0\leq i, j,h \leq D)$, where $I$ (resp. $J$) denotes the identity matrix (resp. the all-ones matrix).
It is known that the matrices $A_0, A_1, \ldots, A_D$ span
a subalgebra $M$ of  ${\rm{Mat}}_{Y}(\mathbb{C})$ and $M$ is called the
{\it{Bose-Mesner algebra}} of $\Gamma$. It turns out that $M$ is generated by the $adjacency\ matrix$ $A_1$.  Since $M$
is commutative and semisimple, it has another basis $E_0, E_1, \ldots, E_D$ called the {\it{primitive idempotents}} of $\Gamma$ (\cite[p. 45]{bcn}). We say that $\Gamma$ is  {\it{$Q$-polynomial}}  with respect to a given ordering
$E_0, E_1, \ldots, E_D$ of the primitive idempotents if there are polynomials $q_i$ of degree $i\ (0\leq i\leq D)$ such that $E_i=q_i(E_1)$ for this ordering,
where the matrix multiplication is entrywise (so that $(E_i)_{xy}=q_i((E_1)_{xy})$ for all vertices
$x,y\in Y$).

Fix a vertex $x\in Y$ and view it as the ``base vertex". For $0\leq i\leq D$, let the diagonal matrix $E^*_i:=E^*_i(x)\in {\rm{Mat}}_{Y}(\mathbb{C})$ denote the $i$-th {\it{dual idempotent}} of $\Gamma$: the
$(y,y)$-entry of $E^*_i$ is $1$ if $\partial(x,y)=i$ and $0$ otherwise. Observe  (\rm{i}) $\overline{E^*_i}^{\rm t}=E^*_i\ (0\leq i \leq D)$; (\rm{ii}) $\sum^D_{i=0}E^*_i=I$; (\rm{iii}) $E^*_iE^*_j=\delta_{ij}E^*_i\ (0\leq i, j \leq D)$.
It is known that $E^*_0, E^*_1, \ldots, E^*_D$ span  a commutative subalgebra $M^*:=M^*(x)$ of ${\rm{Mat}}_{Y}(\mathbb{C})$ and $M^*$ is called the {\it{dual Bose-Mesner algebra}} of $\Gamma$ with respect to $x$ (\cite[p. 378]{p2}).

Let $T:=T(x)$ denote the subalgebra of  ${\rm{Mat}}_{Y}(\mathbb{C})$
generated by $M$ and $M^*$.  The subalgebra $T$ is called
the {\it Terwilliger} (or {\it subconstituent}) {\it algebra} of $\Gamma$ with respect to $x$. This algebra is a
finite-dimensional, semisimple  $\mathbb{C}$-algebra and is non-commutative in general (\cite{p2}).

By a  $T$-{\it{module}}, we mean a subspace $W$ of $V$ such that $bW\subseteq W$
for all $b\in T$.
Let $W,W'$ denote $T$-modules. Then $W,W'$ are said to be $T$-{\it{isomorphic}}  ({\it{isomorphic}} for short) whenever there exists an isomorphism of vector spaces $\phi$: $W\rightarrow W'$ such that
\begin{equation}\nonumber
 (b\phi- \phi b)W=0\ \ \text{ for all $b\in T$}.
\end{equation}
A $T$-module $W$ is said to be $irreducible$ whenever $W\neq 0$ and $W$ contains no $T$-modules other than  $0$ and $W$. Every nonzero $T$-module is an orthogonal direct sum of irreducible $T$-modules; in particular, the standard module $V$ is an orthogonal direct sum of irreducible $T$-modules. We remark that any two non-isomorphic irreducible $T$-modules are
orthogonal.

Let $W$ be an irreducible $T$-module.
The {\it{diameter}}  (resp. {\it{dual diameter}}) of $W$ is defined as $|\{i\mid0\leq i\leq D, E_i^*W\neq 0\}|-1$ (resp. $|\{i\mid 0\leq i\leq D, E_iW\neq 0\}|-1$). $W$ is said to be {\it{thin}} (resp. {\it{dual thin}}) whenever $\dim(E^*_iW)\leq 1$ (resp. $\dim(E_iW)\leq 1$) for all $0\leq i\leq D$. Note that $W$ is thin (resp. dual thin)  if and only if the diameter (resp. dual diameter)  of $W$ is equal to dim$(W)-1$. By the {\it{endpoint}} of $W$, we mean $\min \{i\mid 0\leq i\leq D, E^*_iW\neq 0\}$.
From now on, we suppose $\Gamma$ is $Q$-polynomial with respect to the given ordering
$\{E_i\}^D_{i=0}$.
By the {\it{dual endpoint}} of $W$, we mean  $\min\{i\mid 0\leq i\leq D, E_iW\neq 0\}$).
It is known that $W$ is thin if and only if $W$ is dual thin.

See \cite{p2, ter3} for more information on the Terwilliger algebra. Note that the endpoint (resp.  diameter) of $W$ is called dual endpoint (resp. dual diameter)
of $W$ in the two papers.

\begin{lem}{\rm (\cite[Lemmas 3.9, 3.12]{p2})}\label{lem1}  Let $W$ denote an irreducible $T$-module and
let $\nu, \mu, d\ (0\leq \nu, \mu,  d\leq D) $ denote its endpoint, dual endpoint and diameter, respectively.  The following {\rm{(i)}--\rm{(v)}} hold.
\begin{itemize}
\item[\rm(i)] $A_1E^*_iW\subseteq E^*_{i-1}W+E^*_iW+E^*_{i+1}W$\ \ $(0\leq i\leq D)$.
\item[\rm(ii)] $E^*_iW\neq 0$ if and only if $\nu\leq i\leq \nu+d$.
\item[\rm(iii)] $E^*_jA_1E^*_iW\neq 0$ if $|j-i|=1$\ \ $(\nu\leq i,j \leq \nu+d)$.
\item[\rm(iv)] Suppose $W$ is thin. Then\ \  $E_jW=E_jE^*_\nu W$\ \ $(0\leq j \leq D)$.
\item[\rm(v)] Suppose $W$ is dual thin. Then\ \ $E^*_jW=E^*_jE_\mu W$\ \ $(0\leq j \leq D)$.
\end{itemize}
\end{lem}

\section{The Terwilliger algebra of $O_{m+1}$}
Since $O_{m+1}$ is distance-transitive, its  Terwilliger algebra is (up to isomorphism) independent of the choice of  base vertex. Therefore, for the rest of this paper, we  choose the vertex $x_0:=\{1,2,\ldots,m\}\in X$ as the base vertex and let $T:=T(x_0)$ denote the Terwilliger algebra of $O_{m+1}$ with respect to $x_0$.
In this section, we prove that the algebra $T$
coincides with the centralizer algebra for $O_{m+1}$, and consequently obtain a basis of $T$.
We begin with the following notation.

To each ordered triple $(x,y,z)\in X\times X\times X$,
we associate the four-tuple $(i,j,t,p)$:
\begin{align}\label{eq2}
\partial(x,y,z):=(i,j,t,p),\ \text{where} \ \ i=|x\cap y|,\ j=|x\cap z|,\  t=|y\cap z|,\  p=|x\cap y\cap z|.
\end{align}
Denote by $\mathcal{I}_m$ the set consisting of all four-tuples
$(i,j,t,p)$ that occur as $\partial(x,y,z)=(i,j,t,p)$ for some $x,y,z\in X$.

\begin{pro}\label{pro1}
We have
\begin{align}
\mathcal{I}_m=\big\{(i,j,t,p)\mid\ &0\leq i,j\leq m,\  \mathrm{max}\{i+j-m,m-1-i-j\}\leq t\leq m-|i-j|, \nonumber\\
&\mathrm{max}\{0,i+j-m,i+t-m,j+t-m\}\leq p\leq \label{eq3}\\
&\ \ \ \ \ \ \ \ \ \  \ \ \ \ \ \ \  \ \ \ \ \ \ \ \ \ \  \ \ \ \ \ \ \ \ \ \ \ \ \ \ \ \ \ \ \ \  \mathrm{min}\{i,j,t,i+j+t+1-m\}\big\}\nonumber
\end{align}
for $m\geq 1$. Moreover, the cardinality of $\mathcal{I}_m$ is
${m+4\choose 4}$.
\end{pro}

We remark here that the proof of Proposition \ref{pro1} will be given later.

Let ${\rm Aut}(O_{m+1})$ denote the automorphism group of $O_{m+1}$. The action of  ${\rm Aut}(O_{m+1})$ on $X$ naturally induces the action of  ${\rm Aut}(O_{m+1})$ on $X\times X\times X$:
$\sigma(x,y,z)=(\sigma x,\sigma y,\sigma z)$ for every $\sigma \in {\rm Aut}(O_{m+1})$ and every $(x,y,z)\in X\times X\times X$.
Let ${\rm Aut}_{x_0}(O_{m+1})$ denote the stabilizer of $x_0$ in  ${\rm Aut}(O_{m+1})$. Naturally, we also have the action  of ${\rm Aut}_{x_0}(O_{m+1})$ on $x_0\times X\times X$.
For each $(i,j,t,p)\in \mathcal{I}_m$, we further define
\begin{align}
X_{(i,j,t,p)}=\{(x,y,z)\in X\times X\times X\mid \partial(x,y,z)=(i,j,t,p)\} \
\end{align}
and
\begin{align}
 X^{x_0}_{(i,j,t,p)}=\{(x_0,x,y)\in x_0\times X\times X\mid \partial(x_0,x,y)=(i,j,t,p)\}.\label{eq4}
\end{align}
Observe that $X^{x_0}_{(i,j,t,p)}\subseteq  X_{(i,j,t,p)}$.
In what follows, we  give the meanings of $X_{(i,j,t,p)}$ and $X^{x_0}_{(i,j,t,p)}$, $(i,j,t,p)\in \mathcal{I}_m$.

\begin{pro}\label{pro 15}
The sets {\rm $X_{(i,j,t,p)}$}, $(i,j,t,p)\in \mathcal{I}_m$, are the orbits
of $X\times X\times X$ under the action
of {\rm ${\rm Aut}(O_{m+1})$}.
\end{pro}

\begin{proof}
It is easy to see that the sets {\rm $X_{(i,j,t,p)}$}, $(i,j,t,p)\in \mathcal{I}_m$, form a partition of $X\times X\times X$. By \cite[p. 260]{bcn}, ${\rm Aut}(O_{m+1})$ is sym$(S)$, i.e., it permutes
the $2m+1$ elements in $S$.
Let $x,y,z\in X$ and let $\partial(x,y,z)=(i,j,t,p)$, where $(i,j,t,p)\in \mathcal{I}_m$. By the definitions of $i,j,t$ and $p$, it is easy to see  that $\partial(\sigma x,\sigma y,\sigma z)=(i,j,t,p)$ for any  $\sigma \in {\rm Aut}(O_{m+1})$. This implies $(\sigma x,\sigma y,\sigma z)\in X_{(i,j,t,p)}$ for any  $\sigma \in {\rm Aut}(O_{m+1})$.

To show that {\rm Aut}$(O_{m+1})$ acts transitively on $X_{(i,j,t,p)}$ for each given $(i,j,t,p)\in \mathcal{I}_m$, it suffices to show that for any triple $(x,y,z)$ satisfying $\partial(x,y,z)=(i,j,t,p)$,  there is an automorphism $\sigma\in {\rm Aut}(O_{m+1})$ such that
$(\sigma x,\sigma y,\sigma z)$ is a fixed triple that only depends on $i,j,t$ and $p$. Note
that $|x\cap y|=i,\ |x\cap z|=j,\ |y\cap z|=t\ \text{and}\ |x\cap y\cap z|=p$. Let $A=x\cap y\cap z,\ B=x\cap y-x\cap y\cap z,\ C=x\cap z-x\cap y\cap z,\ D=y\cap z-x\cap y\cap z,\ E=x-y\cup z,\ F=y-x\cup z,\
G=z-x\cup y\ \text{and}\ H=S-x\cup y\cup z$.  We then have $|A|=p,\ |B|=i-p,\ |C|=j-p,\ |D|=t-p,\ |E|=m-i-j+p,\ |F|=m-i-t+p,\ |G|=m-j-t+p\ \text{and}\  |H|=i+j+t-p-m+1$. Moreover, one can readily verify that $x=A\cup B\cup C\cup E$,
$y=A\cup B\cup D\cup F$ and $z=A\cup C\cup D\cup G$.
Pick an automorphism $\sigma\in {\rm Aut}(O_{m+1})$ such that under the action of $\sigma$:
$A\rightarrow\{1,\ldots,p\}$, $B\rightarrow\{p+1,\ldots,i\}$, $C\rightarrow\{i+1,\ldots,i+j-p\}$, $E\rightarrow\{i+j-p+1,\ldots,m\}$, $D\rightarrow\{m+1,\ldots,m+t-p\}$, $F\rightarrow\{m+t-p+1,\ldots,2m-i\}$, $G\rightarrow\{2m-i+1,\ldots,3m-i-j-t+p\}$, $H\rightarrow\{3m-i-j-t+p+1,\ldots,2m+1\}$.
This implies that $\sigma(x)=\sigma(A\cup B\cup C\cup E)=\{1,2,\ldots,m\}$, $\sigma(y)=\sigma(A\cup B\cup D\cup F)=\{1,\ldots,i,m+1,\ldots,2m-i\}$ and $\sigma(z)=\sigma(A\cup C\cup D\cup G)=\{1,\ldots,p,i+1,\ldots,i+j-p,m+1,\ldots,m+t-p,2m-i+1,\ldots,3m-i-j-t+p\}$.
\end{proof}

\begin{pro}\label{cor1}
The sets {\rm $X^{x_0}_{(i,j,t,p)}$}, $(i,j,t,p)\in \mathcal{I}_m$, are the orbits
of $x_0\times X\times X$ under the action
of {\rm ${\rm Aut}_{x_0}(O_{m+1})$}.
\end{pro}

\begin{proof}
Immediate from Proposition \ref{pro 15}.
\end{proof}

\begin{defi}\label{def1}
For each $(i,j,t,p)\in \mathcal{I}_m$, define the matrix $M^{t,p}_{i,j}\in {\rm{Mat}}_X(\mathbb{C})$ by
\begin{equation}\nonumber
(M^{t,p}_{i,j})_{xy}=\left\{\begin{array}{ll} 1 &\text{if}\ (x_0,x, y)\in X^{x_0}_{(i,j,t,p)},\\
 0 &\text{otherwise } \end{array}\right.
\ \ (x, y\in X).
\end{equation}
\end{defi}

Observe that the transpose of $M^{t,p}_{i,j}$ is $M^{t,p}_{j,i}$ and the matrices $M^{t,p}_{i,j},\ (i,j,t,p)\in \mathcal{I}_m$, are linearly independent.
 Moreover, it follows from Proposition \ref{cor1} that each $M^{t,p}_{i,j}$ is invariant under permutating the rows and columns by elements of  ${\rm Aut}_{x_0}(O_{m+1})$.
Let $\mathcal{A}$ be the linear space over $\mathbb{C}$ spanned by the matrices $M^{t,p}_{i,j},\ (i,j,t,p)\in \mathcal{I}_m$. It is known that $\mathcal{A}$ is a matrix
algebra over $\mathbb{C}$ called the {\it{centralizer\ algebra}} (cf. \cite{ban}) of ${\rm Aut}_{x_0}(O_{m+1})$.

For the rest of this paper,
 let $A_1$ and $E^*_i:=E^*_i(x_0)$ $(0\leq i\leq m)$ denote the adjacency  matrix and the  $i$-th dual idempotent of $O_{m+1}$, respectively.
Next, we shall show that the algebras  $T$ and {\rm $\mathcal{A}$}  coincide.

\begin{pro}\label{lem2} With notation as above, the following {\rm (i)--(v)} hold.
\begin{itemize}
\item[\rm (i)] For each $0\leq i\leq m$,
\begin{equation}\label{eq5}
E^*_i=\left\{\begin{array}{ll} M^{m,\frac{2m-i}{2}}_{\frac{2m-i}{2},\frac{2m-i}{2}} &\text{if $i$ is even},\\[.3cm]
 M^{m,\frac{i-1}{2}}_{\frac{i-1}{2},\frac{i-1}{2}} &\text{if $i$ is odd}.
\end{array}\right.
\end{equation}
\item[\rm (ii)] For each $0\leq i\leq m-1$,
\begin{equation}\label{eq6}
E^*_{i+1}A_1E^*_i=\left\{\begin{array}{ll} M^{0,0}_{\frac{i}{2},\frac{2m-i}{2}} &\text{if $i$ is even},\\[.3cm]
 M^{0,0}_{\frac{2m-i-1}{2},\frac{i-1}{2}} &\text{if $i$ is odd}.
\end{array}\right.
\end{equation}
\item[\rm (iii)] For each $0\leq i\leq m-1$,
\begin{equation}\nonumber
E^*_iA_1E^*_{i+1}=\left\{\begin{array}{ll} M^{0,0}_{\frac{2m-i}{2},\frac{i}{2}} &\text{if $i$ is even},\\[.3cm]
 M^{0,0}_{\frac{i-1}{2},\frac{2m-i-1}{2}} &\text{if $i$ is odd}.
\end{array}\right.
\end{equation}
\item[\rm (iv)]
\begin{equation}\nonumber
E^*_{m}A_1E^*_m=M^{0,0}_{\lfloor\frac{m}{2}\rfloor,\lfloor\frac{m}{2}\rfloor}.
\end{equation}
\item[\rm (v)]
\begin{align*}
A_1=\sum^m_{\stackrel{i=0}{i\ even}}(M^{0,0}_{\frac{i}{2},\frac{2m-i}{2}}+M^{0,0}_{\frac{2m-i}{2},\frac{i}{2}})+
\sum^m_{\stackrel{i=0}{i\ odd}}(M^{0,0}_{\frac{2m-i-1}{2},\frac{i-1}{2}}+M^{0,0}_{\frac{i-1}{2},\frac{2m-i-1}{2}})+
M^{0,0}_{\lfloor\frac{m}{2}\rfloor,\lfloor\frac{m}{2}\rfloor}
\end{align*}
\end{itemize}
\end{pro}

\begin{proof}
(i) For even $i\ (0\leq i\leq m)$ and for $x,y\in X$, we consider the $(x,y)$-entries of matrices at both sides of \eqref{eq5}.
By definitions and  \eqref{eq1},
it is easy to see that
$(E^*_i)_{xy}=(M^{m,\frac{2m-i}{2}}_{\frac{2m-i}{2},\frac{2m-i}{2}})_{xy}=1$ if $x=y$, $|x_0\cap x |=\frac{2m-i}{2}$, and 0 otherwise.  This implies \eqref{eq5}  holds for even $i$. Similarly,  for odd $i\ (0\leq i\leq m)$ and for $x,y\in X$,
we have $(E^*_i)_{xy}=(M^{m,\frac{i-1}{2}}_{\frac{i-1}{2},\frac{i-1}{2}})_{xy}=1$ if $x=y$, $|x_0\cap x |=\frac{i-1}{2}$, and 0 otherwise. This implies \eqref{eq5}  holds for odd $i$.

(ii) Similar to the proof of (i): for even $i\ (0\leq i\leq m-1)$ and for $x,y\in X$, we have
$(E^*_{i+1}A_1E^*_i)_{xy}=(M^{0,0}_{\frac{i}{2},\frac{2m-i}{2}})_{xy}=1$ if $|x_0\cap x |=\frac{i}{2}$,$|x_0\cap y|=\frac{2m-i}{2}$, $x\cap y=\emptyset$, and 0 otherwise; for odd $i\ (0\leq i\leq m-1)$ and for $x,y\in X$, we have
$(E^*_{i+1}A_1E^*_i)_{xy}=( M^{0,0}_{\frac{2m-i-1}{2},\frac{i-1}{2}})_{xy}=1$ if $|x_0\cap x |=\frac{2m-i-1}{2}$,$|x_0\cap y|=\frac{i-1}{2}$, $x\cap y=\emptyset$, and 0 otherwise.

(iii) Take transpose of matrices at  both sides of \eqref{eq6}.

(iv) Similar to the proof of (ii). Note that the matrix
$M^{0,0}_{\lfloor\frac{m}{2}\rfloor,\lfloor\frac{m}{2}\rfloor}$ is equal to $M^{0,0}_{\frac{m}{2},\frac{m}{2}}$ if $m$ is even, and  equal to $M^{0,0}_{\frac{m-1}{2},\frac{m-1}{2}}$ if $m$ is odd.

(v)  Note that  $O_{m+1}$ is almost-bipartite. From the equation
\begin{align*}
A_1=\big(\sum^m_{i=0}E^*_i\big)A_1\big(\sum^m_{i=0}E^*_i\big)=\sum^{m-1}_{i=0}(E^*_{i+1}A_1E^*_i+E^*_iA_1E^*_{i+1})+
E^*_mA_1E^*_m
\end{align*}
and (ii)--(iv) above, the result follows.
\end{proof}

\begin{lem}\label{lem3}
The algebra $T$ is a subalgebra of $\mathcal{A}$.
\end{lem}

\begin{proof}
 Immediate from Proposition \ref{lem2}(i),(v) since $T$ is generated by the matrices $A_1$, $E^*_0,E^*_1,\ldots,E^*_m$.
\end{proof}

\begin{lem}{\rm (\cite[Corollary 3.7]{klw})}\label{lem6}
The dimension of $T$ is ${m+4\choose 4}$ for $m\geq 1$.
\end{lem}

To show that $T$ is the same as $\mathcal{A}$,
it suffices to show that the dimension of $\mathcal{A}$ is also ${m+4\choose 4}$ by lemmas \ref{lem3} and \ref{lem6}.  From the above discussion, we know that the dimension of $\mathcal{A}$ is clearly $|\mathcal{I}_m|$ and therefore we now turn to Proposition {\rm \ref{pro1}} and give its proof below.

\begin{proof}[\bfseries\rm\bf Proof of Proposition \ref{pro1}]
Let $\mathcal{I}_m'$ denote the set on the right-hand side of \eqref{eq3}. In the following,   we first  show $\mathcal{I}_m\subseteq\mathcal{I}_m'$.
For each $(i,j,t,p)\in \mathcal{I}_m$, observe that
\begin{align}\label{eq8}
0\leq p\leq i,j,t\leq m
\end{align}
by \eqref{eq2}.
Let $x,y,z\in X$ and let $\partial(x,y,z)=(i,j,t,p)$. Then we have
\begin{align*}
\partial_H(x,y)=2m-2i,\ \partial_H(x,z)=2m-2j,\ \partial_H(y,z)=2m-2t,
\end{align*}
where $\partial_H(u,v)=2m-2|u\cap v|$ denotes the Hamming\ distance between  $u$ and $v\ (u,v\in X)$. From the two inequalities $\partial_H(y,z)\leq \partial_H(x,y)+\partial_H(x,z)$  and $|\partial_H(x,y)-\partial_H(x,z)|\leq \partial_H(y,z)$, it follows that
$i+j-m\leq t\leq m-|i-j|$.
Moreover,
it is easy to see that $|y\cup z|\leq |S|-(|x|-i-j)$ if $i+j\leq m-1$, which implies
$m-1-i-j\leq t$.
Combine the above two inequalities involving $t$ to obtain
\begin{align}\label{eq9}
\mathrm{max}\{i+j-m,m-1-i-j\}\leq t\leq m-|i-j|.
\end{align}
By using the three inequalities
\begin{align*}
&i-p=|x\cap y-z|\leq |x-z|=m-j,\\
&i-p=|x\cap y-z|\leq |y-z|=m-t,\\
&j-p=|x\cap z-y|\leq |z-y|=m-t,
\end{align*}
we obtain $\mathrm{max}\{i+j-m,i+t-m,j+t-m\}\leq p$.
Moreover, we have $p\leq i+j+t+1-m$ since $|x\cup y\cup z|\leq 2m+1$. Combine the above two inequalities involving $p$ to obtain
\begin{align}\label{eq12}
\mathrm{max}\{i+j-m,i+t-m,j+t-m\}\leq p\leq i+j+t+1-m.
\end{align}
 By \eqref{eq8}--\eqref{eq12}, we have $(i,j,t,p)\in \mathcal{I}'_m$ for each $(i,j,t,p)\in \mathcal{I}_m$
 and therefore
\begin{align*}
\mathcal{I}_m\subseteq\mathcal{I}_m'.
\end{align*}

We have shown $\mathcal{I}_m\subseteq\mathcal{I}'_m$ in the above. Next, we  show
\begin{equation}\label{eq13}
|\mathcal{I}'_m|={m+4\choose 4}\ \ \text{for}\ \ m\geq 1
\end{equation}
by induction on $m$.
It is easy to compute $|\mathcal{I}'_1|=5$, which implies that \eqref{eq13} holds for $m=1$. We assume
that \eqref{eq13} holds for $m=k-1\ (k\geq 2)$, that is,
\begin{equation}\label{eq14}
|\mathcal{I}'_{k-1}|={k+3\choose 4}.
\end{equation}
To compute the cardinality of $\mathcal{I}'_{k}$, we define some subsets of $\mathcal{I}'_{k}$ as follows: for each $i\ (0\leq i\leq k)$ and each $l\ (1\leq l\leq i)$, let
\begin{align*}
&\mathcal{B}_{i,i}=\{(i,i,t,p)\mid(i,i,t,p)\in \mathcal{I}'_{k}\}, \\
&\mathcal{B}_{i,i-l}=\{(i,i-l,t,p)\mid(i,i-l,t,p)\in \mathcal{I}'_{k}\},\\
& \mathcal{B}_{i-l,i}=\{(i-l,i,t,p)\mid(i-l,i,t,p)\in \mathcal{I}'_{k}\}.
\end{align*}
Observe that $|\mathcal{B}_{i,i-l}|=|\mathcal{B}_{i-l,i}|$ and all these  subsets are pairwise disjoint. Hence, we have
\begin{align}\label{eq15}
|\mathcal{I}'_k|=\sum^k_{i=0}|\mathcal{B}_{i,i}|+\sum^k_{i=1}\sum^i_{l=1}|\mathcal{B}_{i,i-l}|+
\sum^k_{i=1}\sum^i_{l=1}|\mathcal{B}_{i-l,i}|.
\end{align}
By the definition of $\mathcal{I}'_k$, it is not difficult to compute
\begin{align*}
|\mathcal{B}_{i,i}|=\left\{\begin{array}{ll} (i+1)(i+2) &\text{if $0\leq i\leq \lfloor\frac{k-1}{2}\rfloor$},\\[0.1cm]
(k+1-i)^2 &\text{if $\lfloor\frac{k-1}{2}\rfloor+1\leq i\leq k$}.
 \end{array}\right.
\end{align*}
Then we obtain
\begin{align}\label{eq18}
\sum^k_{i=0}|\mathcal{B}_{i,i}|=\left\{\begin{array}{ll} \frac{(k+1)(k+3)(2k+7)}{24} &\text{if $k$ is odd},\\[0.1cm]
\frac{(k+2)(k+4)(2k+3)}{24} &\text{if $k$ is even}.
 \end{array}\right.
\end{align}
Furthermore, use the definition of $\mathcal{I}'_k$ again to verify that for each $i\ (1\leq i\leq k)$ and each $l\ (1\leq l\leq i)$, both
\begin{align*}
&|\mathcal{B}_{i,i-l}|=|\{(i-1,i-l,t,p):(i-1,i-l,t,p)\in \mathcal{I}'_{k-1}\}|,\ \ \ \ \ \\[.2cm]
&|\mathcal{B}_{i-l,i}|=|\{(i-l,i-1,t,p):(i-l,i-1,t,p)\in \mathcal{I}'_{k-1}\}|.
\end{align*}
From the two equations above, it follows that
\begin{align}\label{eq17}
\sum^k_{i=1}\sum^i_{l=1}|\mathcal{B}_{i,i-l}|+
\sum^k_{i=1}\sum^i_{l=1}|\mathcal{B}_{i-l,i}|=|\{(i,i,t,p):(i,i,t,p)\in\mathcal{I}'_{k-1}\}|+|\mathcal{I}'_{k-1}|.
\end{align}
Note that the value of $|\{(i,i,t,p):(i,i,t,p)\in\mathcal{I}'_{k-1}\}|$ can be computed directly by replacing $k$ by $k-1$ in \eqref{eq18}. Combine \eqref{eq14}, \eqref{eq15}, \eqref{eq18} and \eqref{eq17}, we have
\begin{align}
|\mathcal{I}'_k|
&=\sum^k_{i=0}|\mathcal{B}_{i,i}|+|\{(i,i,t,p):(i,i,t,p)\in\mathcal{I}'_{k-1}\}|+|\mathcal{I}'_{k-1}| \nonumber\\[0.1cm]
&=\frac{(k+1)(k+2)(k+3)}{6}+{k+3\choose 4}\ \ \ \ \ \nonumber\\
&={k+4\choose 4},\nonumber
\end{align}
as desired.

Now, we claim that $\mathcal{I}_m$ is the same as $\mathcal{I}'_m$. Suppose $\mathcal{I}_m$ is
a proper subset of $\mathcal{I}'_m$ for a contradiction. Then we have
$\text{dim}(T)\leq \text{dim}(\mathcal{A})=|\mathcal{I}_m|<{m+4\choose 4}$ by
Lemma \ref{lem3} and \eqref{eq13}. This contradicts Lemma \ref{lem6} and hence our claim holds.

This completes the proof of Proposition \ref{pro1}.
\end{proof}
\begin{thm}\label{thm1}
The algebras {\rm $\mathcal{A}$}  and $T$ coincide.
\end{thm}

\begin{proof}
On the one hand, we have  $T\subseteq \mathcal{A}$
by Lemma \ref{lem3}. On the other hand, we have that $\text{dim}(T)=\text{dim}(\mathcal{A})={m+4\choose 4}$ by Proposition \ref{pro1} and Lemma \ref{lem6}. So the result holds.
\end{proof}

\begin{cor}\label{cor2}
The set $\{M^{t,p}_{i,j}\mid (i,j,t,p)\in \mathcal{I}_m\}$ gives a basis of $T$.
\end{cor}

\begin{proof}
Immediate form Theorem \ref{thm1}.
\end{proof}

\section{The decomposition of $T$}

Recall the standard module $V:=\mathbb{C}^X$.
By a {\it {homogeneous component}} of $V$, we mean  a  nonzero subspace of $V$ spanned by the  irreducible $T$-modules that are isomorphic. In this section, we aim
to give a decomposition of $T$ (in a block-diagonalization form)
by using all the homogeneous components of $V$; this work  is motivated by the following known fact: for a matrix $\ast$-algebra over $\mathbb{C}$ with the identity matrix, all its  elements can be simultaneously block-diagonalized by
some unitary matrix  (see \cite{bek} for more details).

For the rest of this paper, we assume that $m\geq 3$ and that $O_{m+1}$ is $Q$-polynomial with respect to the given ordering
$E_0,E_1,\ldots,E_m$ of its primitive idempotents. The following properties on irreducible $T$-modules are taken from the paper  by J.S. Caughman $et\ al$. \cite{cau}.

\begin{lem}{\rm (\cite{cau})}\label{lem4}
Let $W$ denote an irreducible $T$-module and
let $\nu, \mu, d\ (0\leq \nu, \mu,  d\leq m) $ denote its endpoint, dual endpoint and diameter, respectively. Then the following  {\rm{(i)--(iv)}} hold.
\begin{itemize}
\item[\rm(i)] $W$ is thin and dual thin.
\item[\rm(ii)] $\nu+d=m$.
\item[\rm(iii)] The pair $(\mu, d)$ is restricted to the set below:
\begin{align}\label{eq7}
 \Upsilon:=\{(\mu,d)\in \mathbb{Z}^2|\ 0\leq d\leq m,\ \frac{1}{2}(m-d)\leq\mu\leq m-d \}.
\end{align}
\item[\rm(iv)] The isomorphism class of $W$  depends only on the pair $(\mu,d)$.
\end{itemize}
\end{lem}

Let $W$ denote an irreducible $T$-module with dual endpoint $\mu$ and diameter $d$, where $(\mu,d)\in  \Upsilon$. Define $\mathcal{W}_{(\mu,d)}$ to be the subspace of $V$ spanned by the irreducible $T$-modules that are isomorphic to $W$. Clearly, $\mathcal{W}_{(\mu,d)}$ is a homogeneous component of $V$ associated with  $(\mu,d)$. Since any two non-isomorphic irreducible $T$-modules are
orthogonal, we have
\begin{align}\label{eq56}
V=\sum_{(\mu, d)\in \Upsilon}\mathcal{W}_{(\mu,d)}\ \ \ \  \ (\text {orthogonal\ direct\ sum})
\end{align}
Write $V$  as an orthogonal direct sum of irreducible $T$-modules. By the {\it multiplicity} with
which $W$ appears in $V$, we mean the number of irreducible $T$-modules in this sum which are isomorphic to $W$.
In view of Lemma \ref{lem4}(iv),  we use  $m(\mu,d)$ to denote the multiplicity of $W$.
A formula on $m(\mu,d)$  was given by \cite[Theorem 16.6]{cau} and it allows us to recursively compute
$m(\mu,d)$ for every $(\mu, d)\in \Upsilon$.
However, this Theorem 16.6 does not claim that every pair $(\mu, d)$ in
$\Upsilon$ arises from an irreducible $T$-module.

The following result gives all feasible dual endpoints and diameters for irreducible $T$-modules, and tells us  $m(\mu,d)\neq 0$ for every $(\mu, d)\in\Upsilon$.

\begin{thm}\label{lem5}
There exists an irreducible $T$-module with dual endpoint $\mu$ and diameter $d$ if and only if  $(\mu,d)\in  \Upsilon$.
\end{thm}

\begin{proof}
Lemma \ref{lem4}(iii) implies the lemma in one direction. We next prove the lemma in the other direction. Let $\Psi$ be the subset of  $ \Upsilon$ containing all dual endpoints and diameters that arise from irreducible $T$-modules.  Suppose $\Psi$ is a proper subset of  $ \Upsilon$ for a contradiction.
By Lemma \ref{lem4} and \cite[pp. 96--98]{bek}, we have that there exists a unitary ${2m+1\choose m}\times {2m+1\choose m}$ matrix $U$, whose
columns consist of
appropriate orthonormal bases of all the homogeneous components of $V$ in \eqref{eq56},
such that $\overline{U}^{\rm t}TU$ consists of all block-diagonal matrices:
\begin{equation}\label{eq11}
\bigoplus_{(\mu,d)\in \Psi}m(\mu,d)\odot B_{(\mu,d)},
\end{equation}
where $B_{(\mu,d)}\in \mathbb{C}^{(d+1)\times (d+1)}$ and $m(\mu,d)\odot B_{(\mu,d)}$ denotes the iterated direct sum $\oplus^{m(\mu,d)}_{i=1}B_{(\mu,d)}$.
Then by deleting copies of blocks in \eqref{eq11} and using $\lfloor\frac{m-d+1}{2}\rfloor\leq \mu\leq m-d$ from \eqref{eq7}, we obtain the following inequality
\begin{align*}
{\rm dim}(T)=\sum_{(\mu,d)\in \Psi}(d+1)^2<\sum_{(\mu,d)\in  \Upsilon}(d+1)^2
&=\sum^m_{d=0}(m-d-\lfloor\frac{m-d+1}{2}\rfloor+1)(d+1)^2\\
&={m+4\choose 4}.
\end{align*}
This contradicts the fact that ${\rm dim}(T)={m+4\choose 4}$. So $\Psi= \Upsilon$. Thus the result holds.
\end{proof}

\begin{thm}\label{pro5}
The algebra $T$ is isomorphic to
\begin{equation}
\bigoplus^{m}_{d=0}(m-d-\lfloor\frac{m-d+1}{2}\rfloor+1)\odot\mathbb{C}^{(d+1)\times (d+1)}.\nonumber
\end{equation}
\end{thm}

\begin{proof}
Immediate from Theorem \ref{lem5} and its proof.
\end{proof}

\begin{lem}\label{lem21}
The center of $T$ has dimension $\lfloor\frac{(m+2)^2}{4}\rfloor$.
\end{lem}

\begin{proof}
The dimension of the center of $T$ is exactly the number of isomorphism classes for irreducible $T$-modules. By Lemma \ref{lem4} and Theorem \ref{lem5}, this number is $|\Upsilon|$ that equals $\frac{(m+1)(m+3)}{4}$ if $m$ is odd, and $\frac{(m+2)^2}{4}$ if $m$ is even.
\end{proof}

\section{The homogeneous components of $V$}

For any fixed pair $(\mu,d)\in \Upsilon$, recall $\mathcal{W}_{(\mu,d)}$ is a homogeneous component of $V$ associated with $(\mu,d)$.
In this section, we display an orthogonal basis of $\mathcal{W}_{(\mu,d)}$. We begin with the following lemma which is based on Proposition \ref{lem2}(i), (ii).

\begin{lem}\label{lem7}
 For  nonnegative integers $i,k\ (0\leq i,k \leq m)$ satisfying $i+k\leq m$, the following {\rm (i)--(iv)} hold.
\begin{itemize}
\item[\rm (i)] If $i$ is even and $k$ is odd, then
\begin{align}\label{eq21}
E^*_{i+k}A_1E^*_{i+k-1}A_1E^*_{i+k-2}\cdots E^*_{i+1}A_1E^*_{i}
=\big((\frac{k-1}{2})!\big)^2\frac{k+1}{2} M^{\frac{k-1}{2},\frac{k-1}{2}}_{\frac{i+k-1}{2},\frac{2m-i}{2}}.
\end{align}
\item[\rm (ii)] If $i$ is even and $k\ (k\geq 2)$ is even, then
\begin{align}\label{eq16}
E^*_{i+k}A_1E^*_{i+k-1}A_1E^*_{i+k-2}\cdots E^*_{i+1}A_1E^*_{i}
=\big((\frac{k}{2})!\big)^2 M^{\frac{2m-k}{2},\frac{2m-i-k}{2}}_{\frac{2m-i-k}{2},\frac{2m-i}{2}}.
\end{align}
\item[\rm (iii)] If $i$ is odd and $k$ is odd, then
\begin{align*}
E^*_{i+k}A_1E^*_{i+k-1}A_1E^*_{i+k-2}\cdots E^*_{i+1}A_1E^*_{i}
=\big((\frac{k-1}{2})!\big)^2\frac{k+1}{2} M^{\frac{k-1}{2},0}_{\frac{2m-i-k}{2},\frac{i-1}{2}}.
\end{align*}
\item[\rm (iv)] If $i$ is odd and $k\ (k\geq 2)$ is even, then
\begin{align*}
E^*_{i+k}A_1E^*_{i+k-1}A_1E^*_{i+k-2}\cdots E^*_{i+1}A_1E^*_{i}
=\big((\frac{k}{2})!\big)^2 M^{\frac{2m-k}{2},\frac{i-1}{2}}_{\frac{i+k-1}{2},\frac{i-1}{2}}.
\end{align*}
\end{itemize}
\end{lem}

\begin{proof}
(i) For $i$ being even and $k$ being odd, we use induction on $k\ (k\geq 1)$ to show \eqref{eq21}. By \eqref{eq6},
 it is easy to see \eqref{eq21} holds for $k=1$. We now assume that
\eqref{eq21} holds for $k-2\ (k\geq 3)$, that is,
\begin{align}\label{eq22}
E^*_{i+k-2}A_1E^*_{i+k-3}A_1E^*_{i+k-4}\cdots E^*_{i+1}A_1E^*_{i}
=\big((\frac{k-3}{2})!\big)^2\frac{k-1}{2} M^{\frac{k-3}{2},\frac{k-3}{2}}_{\frac{i+k-3}{2},\frac{2m-i}{2}}.
\end{align}
 Observe that $E^*_{i+k-1}A_1E^*_{i+k-2}=M^{0,0}_{\frac{2m-i-k+1}{2},\frac{i+k-3}{2}}$
 and $E^*_{i+k}A_1E^*_{i+k-1}=M^{0,0}_{\frac{i+k-1}{2},\frac{2m-i-k+1}{2}}$ by \eqref{eq6}.
 Moreover, one can readily  verify that
\begin{align}\label{eq23}
M^{0,0}_{\frac{2m-i-k+1}{2},\frac{i+k-3}{2}}M^{\frac{k-3}{2},\frac{k-3}{2}}_{\frac{i+k-3}{2},\frac{2m-i}{2}}
&=\frac{k-1}{2} M^{\frac{2m-k+1}{2},\frac{2m-i-k+1}{2}}_{\frac{2m-i-k+1}{2},\frac{2m-i}{2}}
\end{align}
since the $(x,y)$-entry of  left-side of \eqref{eq23}, with $|x\cap x_0|=\frac{2m-i-k+1}{2}$
and $|y\cap x_0|=\frac{2m-i}{2}$, is equal to $|\{z\in X:|z\cap x_0|=\frac{i+k-3}{2},z\cap x=\emptyset,|z\cap y|=|z\cap y\cap x_0|=\frac{k-3}{2}\}|=\frac{k-1}{2}$, and that
\begin{align}\label{eq24}
M^{0,0}_{\frac{i+k-1}{2},\frac{2m-i-k+1}{2}} M^{\frac{2m-k+1}{2},\frac{2m-i-k+1}{2}}_{\frac{2m-i-k+1}{2},\frac{2m-i}{2}}
=\frac{k+1}{2} M^{\frac{k-1}{2},\frac{k-1}{2}}_{\frac{i+k-1}{2},\frac{2m-i}{2}}
\end{align}
since the $(x,y)$-entry of  left-side of \eqref{eq24}, with $|x\cap x_0|=\frac{i+k-1}{2}$
and $|y\cap x_0|=\frac{2m-i}{2}$, is equal to $|\{z\in X:|z\cap x_0|=\frac{2m-i-k+1}{2},z\cap x=\emptyset,|z\cap y|=\frac{2m-k+1}{2},|z\cap y\cap x_0|=\frac{2m-i-k+1}{2}\}|=\frac{k+1}{2}$.
Then combine equations
\eqref{eq22}--\eqref{eq24} to obtain the equation \eqref{eq21}, as desired.

(ii) For $i$ being even and $k\ (k\geq 2)$ being even, it follows from (i) that
\begin{align}\label{eq25}
E^*_{i+k-1}A_1E^*_{i+k-2}A_1E^*_{i+k-3}\cdots E^*_{i+1}A_1E^*_{i}
=\big((\frac{k-2}{2})!\big)^2\frac{k}{2} M^{\frac{k-2}{2},\frac{k-2}{2}}_{\frac{i+k-2}{2},\frac{2m-i}{2}}.
\end{align}
Observe that $E^*_{i+k}A_1E^*_{i+k-1}=M^{0,0}_{\frac{2m-i-k}{2},\frac{i+k-2}{2}}$ by \eqref{eq6}.  Moreover, we have
\begin{align}\label{eq26}
M^{0,0}_{\frac{2m-i-k}{2},\frac{i+k-2}{2}}M^{\frac{k-2}{2},\frac{k-2}{2}}_{\frac{i+k-2}{2},\frac{2m-i}{2}}
=\frac{k}{2} M^{\frac{2m-k}{2},\frac{2m-i-k}{2}}_{\frac{2m-i-k}{2},\frac{2m-i}{2}}
\end{align}
since the $(x,y)$-entry of  left-side of \eqref{eq26}, with $|x\cap x_0|=\frac{2m-i-k}{2}$
and $|y\cap x_0|=\frac{2m-i}{2}$, is equal to $|\{z\in X:|z\cap x_0|=\frac{i+k-2}{2},z\cap x=\emptyset,|z\cap y|=|z\cap y\cap x_0|=\frac{k-2}{2}\}|=\frac{k}{2}$.
Then combine equations
\eqref{eq25}, \eqref{eq26} to obtain the equation \eqref{eq16}.

(iii) Similar to the proof of (i): combine the following three equations
\begin{align*}
&E^*_{i+k-2}A_1E^*_{i+k-3}A_1E^*_{i+k-4}\cdots E^*_{i+1}A_1E^*_{i}
=\big((\frac{k-3}{2})!\big)^2\frac{k-1}{2} M^{\frac{k-3}{2},0}_{\frac{2m-i-k+2}{2},\frac{i-1}{2}}\\
&M^{0,0}_{\frac{i+k-2}{2},\frac{2m-i-k+2}{2}}M^{\frac{k-3}{2},0}_{\frac{2m-i-k+2}{2},\frac{i-1}{2}}
=\frac{k-1}{2} M^{\frac{2m-k+1}{2},\frac{i-1}{2}}_{\frac{i+k-2}{2},\frac{i-1}{2}}\\
&M^{0,0}_{\frac{2m-i-k}{2},\frac{i+k-2}{2}}M^{\frac{2m-k+1}{2},\frac{i-1}{2}}_{\frac{i+k-2}{2},\frac{i-1}{2}}
=\frac{k+1}{2} M^{\frac{k-1}{2},0}_{\frac{2m-i-k}{2},\frac{i-1}{2}}
\end{align*}
to obtain the result.

(iv) Similar to the proof of (ii): combine the following two equations
\begin{align*}
&E^*_{i+k-1}A_1E^*_{i+k-2}A_1E^*_{i+k-3}\cdots E^*_{i+1}A_1E^*_{i}
=\big((\frac{k-2}{2})!\big)^2\frac{k}{2} M^{\frac{k-2}{2},0}_{\frac{2m-i-k+1}{2},\frac{i-1}{2}}\\
&M^{0,0}_{\frac{i+k-1}{2},\frac{2m-i-k+1}{2}}M^{\frac{k-2}{2},0}_{\frac{2m-i-k+1}{2},\frac{i-1}{2}}
=\frac{k}{2} M^{\frac{2m-k}{2},\frac{i-1}{2}}_{\frac{i+k-1}{2},\frac{i-1}{2}}
\end{align*}
to obtain the result.
\end{proof}

Let $W$ denote an irreducible $T$-module with endpoint $\nu$ and diameter $d \ (0\leq \nu,d\leq m)$.
If $\nu$ is even, then  we define the vector $\xi_k\in V\ (0\leq k\leq d)$ by
\begin{align}\label{eq27}
\xi_k=\left\{\begin{array}{ll} M^{\frac{k-1}{2},\frac{k-1}{2}}_{\frac{\nu+k-1}{2},\frac{2m-\nu}{2}}\xi_0 &\text{if $k$ is odd,}\\[.2cm] M^{\frac{2m-k}{2},\frac{2m-\nu-k}{2}}_{\frac{2m-\nu-k}{2},\frac{2m-\nu}{2}}\xi_0&\text{if $k$ is even},
\end{array}\right.
\end{align}
where $\xi_0$ is  any nonzero vector in $E^*_{\nu}W$. Note that $M^{m,\frac{2m-\nu}{2}}_{\frac{2m-\nu}{2},\frac{2m-\nu}{2}}\xi_0=E^*_{\nu}\xi_0=\xi_0$ by \eqref{eq5}.

If $\nu$ is odd, then  we define the vector $\eta_k\in V\ (0\leq k\leq d)$ by
\begin{align}\label{eq28}
\eta_k=\left\{\begin{array}{ll} M^{\frac{k-1}{2},0}_{\frac{2m-\nu-k}{2},\frac{\nu-1}{2}}\eta_0 &\text{if $k$ is odd,}\\[.2cm]
M^{\frac{2m-k}{2},\frac{\nu-1}{2}}_{\frac{\nu+k-1}{2},\frac{\nu-1}{2}}\eta_0&\text{if $k$ is even},
\end{array}\right.
\end{align}
where $\eta_0$ is any nonzero vector in $E^*_{\nu}W$. Note that $M^{m,\frac{\nu-1}{2}}_{\frac{\nu-1}{2},\frac{\nu-1}{2}}\eta_0=E^*_{\nu}\eta_0=\eta_0$ by \eqref{eq5}.

\begin{lem}\label{lem12}
Let $W$ denote an irreducible $T$-module with endpoint $\nu$ and diameter $d\ (0\leq \nu,d\leq m)$. Then the following  {\rm{(i)}, \rm{(ii)}} hold.
\begin{itemize}
\item[\rm(i)] If $\nu$ is even, then the  vectors
 $\xi_0,\xi_1,\ldots,\xi_{d}$ of form \eqref{eq27} give an orthogonal basis of $W$.
\item[\rm(ii)] If $\nu$ is odd,
then the  vectors $\eta_0,\eta_1,\ldots,\eta_{d}$ of form \eqref{eq28} give an orthogonal basis of $W$.
\end{itemize}
\end{lem}

\begin{proof}
(i) Pick any nonzero vector $\xi_0\in E^*_{\nu}W$. It follows from \cite[Lemma 2.2]{hhgy} that the $d+1$ nonzero vectors $\xi_0, \ E^*_{\nu+1}A_1E^*_{\nu}\xi_0$, $\ldots$,
$E^*_{\nu+d}A_1E^*_{\nu+d-1}\cdots E^*_{\nu+1}A_1E^*_{\nu}\xi_0$ form an orthogonal basis of $W$.
Then from Lemma \ref{lem7}(i),(ii), our result follows.

(ii) Similar to the proof of (i).
\end{proof}

For each $\nu\ (0\leq \nu\leq m)$, let $\textsf{W}_\nu$ denote the $T$-module spanned by the irreducible $T$-modules having  the same endpoint $\nu$. By Lemma \ref{lem4} and Theorem \ref{lem5}, these irreducible $T$-modules also have the same diameter $d=m-\nu$, and have different dual endpoints $\mu$ in the range $\lfloor \frac{\nu+1}{2}\rfloor\leq \mu\leq \nu$. Therefore, we have
\begin{align}\label{eq32}
\textsf{W}_\nu=\mathcal{W}_{(\lfloor \frac{\nu+1}{2}\rfloor,d)}+\mathcal{W}_{(\lfloor \frac{\nu+1}{2}\rfloor+1,d)}+\cdots+\mathcal{W}_{(\nu,d)}\ \ (\text {orthogonal\ direct\ sum})
\end{align}
and
\begin{align}\label{eq31}
V=\textsf{W}_0+\textsf{W}_1+\cdots+\textsf{W}_m\ \ (\text {orthogonal\ direct\ sum}).
\end{align}

For each $\nu\ (0\leq \nu\leq m)$, we  define a corresponding subspace of $V$ as follows:
\begin{align}
\mathcal{L}_{\nu}=\{\xi\in V\mid E^*_{\nu-1}A_1E^*_{\nu}\xi=0,\ \xi_{x}=0 \ \text{if}\ \partial(x,x_0)\neq \nu\}.
\end{align}
Note that by Proposition \ref{lem2}(iii), $E^*_{\nu-1}A_1E^*_{\nu}$ is $M^{0,0}_{\frac{2m-\nu+1}{2},\frac{\nu-1}{2}}$ if $\nu$ is odd,
and  $M^{0,0}_{\frac{\nu-2}{2},\frac{2m-\nu}{2}}$ if $\nu$ is even.
\begin{lem}\label{lem11}
With notation as above, the following  {\rm{(i),\ (ii)}} hold.

\begin{itemize}
\item[\rm(i)] For each $\nu\ (0\leq \nu\leq m)$,\
$\mathcal{L}_\nu=E^*_\nu${\rm{\textsf{W}}}$_\nu$.
\item[\rm(ii)]
For each $\nu\ (0\leq \nu\leq m)$ and each $\mu\ (\lfloor \frac{\nu+1}{2}\rfloor\leq \mu\leq \nu)$,
\begin{align}\label{eq51}
E^*_\nu\mathcal{W}_{(\mu,d)}=E^*_\nu E_{\mu}P(E_{\mu-1})\mathcal{L}_\nu,
\end{align}
where $d=m-\nu$,
$P(E_{\mu-1})=I$ for  $\mu=\lfloor \frac{\nu+1}{2}\rfloor$ and $P(E_{\mu-1})=(I-E^*_\nu E_{\lfloor \frac{\nu+1}{2}\rfloor})(I-E^*_\nu E_{\lfloor \frac{\nu+1}{2}\rfloor+1})\cdots(I-E^*_\nu E_{\mu-1})$ for  $\lfloor \frac{\nu+1}{2}\rfloor+1\leq \mu\leq \nu$.
\end{itemize}
\end{lem}

\begin{proof}
(i)
It is easy to see  that any nonzero vector $\xi\in \mathcal{L}_\nu$ if and only if
\begin{equation}
E^*_\nu\xi\neq 0,\
   E^*_i\xi=0\ (i\neq \nu)\ \text{and}\ E^*_{\nu-1}A_1E^*_\nu\xi=0 \ \ (0\leq \nu\leq m),\nonumber
\end{equation}
where $E^*_{-1}=0$.
Pick any  nonzero vector $\xi'\in \mathcal{L}_\nu$.
Since $E^*_\nu\xi'\neq 0$ and $E^*_i\xi'=0$ ($i\neq \nu$), we have $\xi'\in  E^*_\nu V$. Then by \eqref{eq31}, we have $\xi'\in  E^*_\nu (\textsf{W}_0+\textsf{W}_1+\cdots+\textsf{W}_\nu)$.
By Lemmas \ref{lem1}(iii), \ref{lem4}(ii) and since $E^*_{\nu-1}A_1E^*_\nu\xi=0$, we have $\xi'\in E^*_\nu\textsf{W}_\nu$. Thus $\mathcal{L}_\nu\subseteq E^*_\nu\textsf{W}_\nu$.

Conversely, pick any nonzero vector $\xi'\in E^*_\nu \textsf{W}_\nu$. It is easy to see that
$E^*_\nu\xi'\neq 0$ and
$E^*_i\xi'=0$ if $i\neq \nu$. Moreover, by Lemma \ref{lem1}(i) we have that $E^*_{\nu-1}A_1E^*_\nu\xi'\in E^*_{\nu-1}(E^*_{\nu-1}\textsf{W}_\nu+E^*_\nu\textsf{W}_\nu+E^*_{\nu+1}\textsf{W}_\nu)=0$.  Therefore, we have $\xi'\in \mathcal{L}_\nu$. This implies
$E^*_\nu\textsf{W}_\nu\subseteq \mathcal{L}_\nu$.

(ii) We first consider the case of $\mu=\lfloor \frac{\nu+1}{2}\rfloor$. Multiply both sides of \eqref{eq32} on the left by $E_{\lfloor \frac{\nu+1}{2}\rfloor}$ to obtain
$E_{\lfloor \frac{\nu+1}{2}\rfloor}\textsf{W}_\nu=E_{\lfloor \frac{\nu+1}{2}\rfloor}\mathcal{W}_{(\lfloor \frac{\nu+1}{2}\rfloor,d)}.$
Then we have
\begin{align}\label{eq35}
E^*_\nu\mathcal{W}_{(\lfloor \frac{\nu+1}{2}\rfloor,d)}&=E^*_\nu E_{\lfloor \frac{\nu+1}{2}\rfloor}\mathcal{W}_{(\lfloor \frac{\nu+1}{2}\rfloor,d)}\ \ \ \ \ \ \ \text{(by Lemma \ref{lem1}(v))}\nonumber\\
&=E^*_\nu E_{\lfloor \frac{\nu+1}{2}\rfloor}\textsf{W}_\nu \ \ \ \ \ \ \ \ \ \ \ \ \ \ \ \ \nonumber\\
&=E^*_\nu E_{\lfloor \frac{\nu+1}{2}\rfloor}E^*_\nu\textsf{W}_\nu \ \ \ \ \ \ \ \ \ \ \ \ \ \text{(by Lemma \ref{lem1}(iv))}\nonumber\\
&=E^*_\nu E_{\lfloor \frac{\nu+1}{2}\rfloor}\mathcal{L}_\nu \ \ \ \ \ \ \ \ \ \ \ \ \ \ \ \ \  \text{(by (i) above).}
\end{align}
This implies the equation \eqref{eq51} holds for $\mu=\lfloor \frac{\nu+1}{2}\rfloor$.
Similarly, for $\mu=\lfloor \frac{\nu+1}{2}\rfloor+1$, multiply both sides of \eqref{eq32} on the left by $E_{\lfloor \frac{\nu+1}{2}\rfloor+1}$ to obtain
$E_{\lfloor \frac{\nu+1}{2}\rfloor+1}\textsf{W}_\nu=E_{\lfloor \frac{\nu+1}{2}\rfloor+1}\mathcal{W}_{(\lfloor \frac{\nu+1}{2}\rfloor,d)}+E_{\lfloor \frac{\nu+1}{2}\rfloor+1}\mathcal{W}_{(\lfloor \frac{\nu+1}{2}\rfloor+1,d)}$.
Then we have
\begin{align}\label{eq36}
E^*_\nu\mathcal{W}_{(\lfloor \frac{\nu+1}{2}\rfloor+1,d)}&=E^*_\nu E_{\lfloor \frac{\nu+1}{2}\rfloor+1}\mathcal{W}_{(\lfloor \frac{\nu+1}{2}\rfloor+1,d)}\ \ \ \ \ \ \ \ \ \ \ \ \ \ \text{(by Lemma \ref{lem1}(v))}\nonumber\\
&=E^*_\nu \big(E_{\lfloor \frac{\nu+1}{2}\rfloor+1}\textsf{W}_\nu-E_{\lfloor \frac{\nu+1}{2}\rfloor+1}\mathcal{W}_{(\lfloor \frac{\nu+1}{2}\rfloor,d)}\big)\nonumber\\
&=E^*_\nu E_{\lfloor \frac{\nu+1}{2}\rfloor+1}\big(E^*_\nu\textsf{W}_\nu-E^*_\nu\mathcal{W}_{(\lfloor \frac{\nu+1}{2}\rfloor,d)}\big)\ \ \ \ \ \ \ \ \text{(by Lemma \ref{lem1}(iv))}\nonumber\\
&=E^*_\nu E_{\lfloor \frac{\nu+1}{2}\rfloor+1}\big(I-E^*_\nu E_{\lfloor \frac{\nu+1}{2}\rfloor}\big)\mathcal{L}_\nu \ \ \ \ \ \ \ \ \text{(by (i) above and \eqref{eq35}).}
\end{align}
This implies the equation \eqref{eq51} holds for $\mu=\lfloor \frac{\nu+1}{2}\rfloor+1$.
For the remaining cases of $\mu$, we assume that the equation \eqref{eq51} holds for $\mu-1\ (\mu\geq \lfloor \frac{\nu+1}{2}\rfloor+2)$, that is,
\begin{align}\label{eq53}
E^*_\nu\mathcal{W}_{(\mu-1,d)}=E^*_\nu E_{\mu-1}P(E_{\mu-2})\mathcal{L}_\nu.
\end{align}
Note that, similar to obtaining \eqref{eq36},  the right-hand side of \eqref{eq53} is a direct result from
\begin{align*}
E^*_\nu E_{\mu-1}\big(E^*_\nu\textsf{W}_\nu-E^*_\nu\mathcal{W}_{(\lfloor \frac{\nu+1}{2}\rfloor,d)}-\cdots-E^*_\nu\mathcal{W}_{(\mu-2,d)}\big)
\end{align*}
by our assumption. Then by using a similar argument used to obtain \eqref{eq36}, we have
\begin{align*}
E^*_\nu\mathcal{W}_{(\mu,d)}&=E^*_\nu E_\mu\mathcal{W}_{(\mu,d)}\nonumber\\
&=E^*_\nu E_\mu\big(E^*_\nu\textsf{W}_\nu-E^*_\nu\mathcal{W}_{(\lfloor \frac{\nu+1}{2}\rfloor,d)}-\cdots-E^*_\nu\mathcal{W}_{(\mu-2,d)}-E^*_\nu\mathcal{W}_{(\mu-1,d)}\big)\nonumber\\
&=E^*_\nu E_\mu\big(P(E_{\mu-2})\mathcal{L}_\nu-E^*_\nu E_{\mu-1}P(E_{\mu-2})\mathcal{L}_\nu\big)\ \ \ \ \ \ \ \ \ \ \ \text{(by \eqref{eq53}).}\nonumber\\
&=E^*_\nu E_\mu P(E_{\mu-1})\mathcal{L}_\nu,
\end{align*}
as desired.
\end{proof}

For each $\nu\ (0\leq \nu\leq m)$ and each $\mu\ (\lfloor \frac{\nu+1}{2}\rfloor\leq \mu \leq \nu)$, based on Lemma \ref{lem11}(ii),
let the set
$\Lambda_{(\mu,m-\nu)}$ denote an orthogonal basis  of $E^*_\nu E_{\mu}P(E_{\mu-1})\mathcal{L}_\nu\ (=E^*_\nu\mathcal{W}_{(\mu,d)})$. Note that
$|\Lambda_{(\mu,m-\nu)}|=m(\mu,m-\nu)$. Then for each $\xi\in \Lambda_{(\mu,m-\nu)}$ and each $k\ (\nu\leq k\leq m)$, we define the vector $b_{(\nu,\mu,\xi,k)}\in V$ by
\begin{align}\label{eq02}
b_{(\nu,\mu,\xi,k)}=\left\{\begin{array}{ll}
M^{\frac{k-\nu-1}{2},\frac{k-\nu-1}{2}}_{\frac{k-1}{2},\frac{2m-\nu}{2}}\xi &\text{if  $\nu$  is even and $k-\nu$ is odd,}\\[.2cm]
 M^{\frac{2m-k+\nu}{2},\frac{2m-k}{2}}_{\frac{2m-k}{2},\frac{2m-\nu}{2}}\xi &\text{if  $\nu$  is even and $k-\nu$ is even,}
\end{array}\right.
\end{align}
and
\begin{align}\label{eq01}
b_{(\nu,\mu,\xi,k)}=\left\{\begin{array}{ll}
M^{\frac{k-\nu-1}{2},0}_{\frac{2m-k}{2},\frac{\nu-1}{2}}\xi &\text{if  $\nu$  is odd and $k-\nu$ is odd,}\\[.2cm]
 M^{\frac{2m-k+\nu}{2},\frac{\nu-1}{2}}_{\frac{k-1}{2},\frac{\nu-1}{2}}\xi &\text{if  $\nu$  is odd and $k-\nu$ is even.}
\end{array}\right.
\end{align}
 Note that the vectors $b_{(\nu,\mu,\xi,k)}$ of form \eqref{eq02} (resp. \eqref{eq01}) are given based on \eqref{eq27} (resp. \eqref{eq28}) by replacing $k$ by $k-\nu$.

\begin{thm}\label{pro2}
For a fixed pair $(\mu,m-\nu)\in\Upsilon$, let  $\mathcal{W}_{(\mu,m-\nu)}$ denote the homogeneous component of $V$ associated with $(\mu,m-\nu)$. Then the set $\{b_{(\nu,\mu,\xi,k)}\mid\xi\in \Lambda_{(\mu,m-\nu)},\nu\leq k\leq m\}$
forms an orthogonal basis of $\mathcal{W}_{(\mu,m-\nu)}$.
\end{thm}

\begin{proof}
For a fixed pair $(\mu,m-\nu)\in  \Upsilon$ and fixed vector $\xi\in \Lambda_{(\mu,m-\nu)}$,  it follows from Lemma \ref{lem12} that the  vectors $b_{(\nu,\mu,\xi,k)}\ (k=\nu, \nu+1,\ldots,m)$
 give an orthogonal basis of some irreducible $T$-module, denoted by $W_\xi$, with dual endpoint $\mu$ and diameter $m-\nu$.
 Moreover, it is easy to see that   $\mathcal{W}_{(\mu,m-\nu)}$ is the orthogonal direct sum of all $W_\xi$, $\xi\in \Lambda_{(\mu,m-\nu)}$. Therefore, the set $\{b_{(\nu,\mu,\xi,k)}\mid\xi\in \Lambda_{(\mu,m-\nu)},\nu\leq k\leq m\}$
forms an orthogonal basis of $\mathcal{W}_{(\mu,m-\nu)}$.
\end{proof}

\begin{cor}\label{cor3}
The  set $\{b_{(\nu,\mu,\xi,k)}\mid (\mu,m-\nu)\in\Upsilon,
\xi\in \Lambda_{(\mu,m-\nu)},\nu\leq k\leq m\}$
forms  an orthogonal
basis of $V$.
\end{cor}

\begin{proof}
Immediate from the equations \eqref{eq32}, \eqref{eq31} and Theorem \ref{pro2}.
\end{proof}

\section*{Acknowledgement}

This work is supported by the NSF of China (No. 11971146 and No. 12101175), the NSF of Hebei Province (No. A2019205089 and No. A2020403024).

\end{document}